\documentclass[paper,10pt]{article}
\usepackage[utf8]{inputenc}
\usepackage{amsmath}
\usepackage{amsthm}
\usepackage{amssymb}
\usepackage{xcolor}
\usepackage{comment}

\usepackage{amsfonts, latexsym, fancyhdr, graphicx, subcaption}
\usepackage{epstopdf} 
\usepackage{epsfig} 

\usepackage{tikz}
\usetikzlibrary{
intersections, arrows.meta,
automata,er,calc,
backgrounds,
mindmap,folding,
patterns,
decorations.markings,
fit,
shapes,matrix,
positioning,
shapes.geometric,
arrows,through
}

\newtheorem{theorem}{Theorem}
\newtheorem{lemma}{Lemma}

\newcommand{\mostafa}[1]{\textcolor{red}{#1}}
\newcommand{\mattia}[1]{\textcolor{cyan}{#1}}

\newtheorem{corollary}{Corollary}

\newtheorem{proposition}{Proposition}
\newtheorem{conjecture}{Conjecture}

\theoremstyle{definition}

\title{A multigroup approach to delayed prion production}

\author{Mostafa Adimy$^{1}$, Abdennasser Chekroun$^{2}$, Laurent Pujo-Menjouet$^{3}$,\\ Mattia Sensi$^{4,5,*}$\\[1em]
$^1${\footnotesize Inria, Univ Lyon, Université Claude Bernard Lyon 1, CNRS UMR 5208, Institut Camille Jordan,}\\{\footnotesize F-69603 Villeurbanne, France}\\[0.5em]
$^2${\footnotesize Laboratoire d'Analyse Nonlin\'{e}aire et Math\'{e}matiques Appliqu\'{e}es, University of Tlemcen, Tlemcen 13000, Algeria}\\[0.5em]
$^3${\footnotesize Univ Lyon, Inria, Université Claude Bernard Lyon 1, CNRS UMR 5208, Institut Camille Jordan,}\\{\footnotesize F-69603 Villeurbanne, France}\\[0.5em]
$^4${\footnotesize MathNeuro Team, Inria at Universit\'e C\^ote d'Azur, 2004 Rte des Lucioles, 06410 Biot, France},\\
$^5${\footnotesize Department of Mathematical Sciences ``G. L. Lagrange'', Politecnico di Torino,}\\ \footnotesize{Corso Duca degli Abruzzi 24, 10129 Torino Italy}\\$^*${\footnotesize Corresponding author: \texttt{mattia.sensi@polito.it}}}
\date{\today}

\begin{document}

\maketitle

\begin{abstract}
    We generalize the model proposed in [Adimy, Babin, Pujo-Menjouet, \emph{SIAM Journal on Applied Dynamical Systems} (2022)] for prion infection to a network of neurons. We do so by applying a so-called \emph{multigroup approach} to the system of Delay Differential Equations (DDEs) proposed in the aforementioned paper. We derive the classical threshold quantity $\mathcal{R}_0$, \textit{i.e.} the basic reproduction number, exploiting the fact that the DDEs of our model qualitatively behave like Ordinary Differential Equations (ODEs) when evaluated at the Disease Free Equilibrium. We prove analytically that the disease naturally goes extinct when $\mathcal{R}_0<1$, whereas it persists when $\mathcal{R}_0>1$. We conclude with some selected numerical simulations of the system, to illustrate our analytical results.
    %\mattia{Submission to: Discrete and Continuous Dynamical Systems - B}
\end{abstract}

\section{Introduction}

%How to comment: \mattia{Mattia} \laurent{Laurent} \mostafa{Mostafa} \abdennasser{Abdennasser} 

Prion is a protein involved in neurodegenerative diseases and more 
particularly the transmissible spongiform encephalopathies such as scrapie
for sheep,  bovine spongiform encephalopathy, also known as mad cow disease
in cattle, and the Creutzfeldt-Jakob disease in humans 
\cite{prusiner_prion_1998,roucou_cellular_2005}. Produced by the cells, this
protein in its normal form is called PrP$^C$ (for Prion Protein Cellular)
and appears to be protective \cite{roucou_cellular_2005}. However, 
it becomes harmful and fatal  when its shape changes. This misfolded pathological  conformation also known as PrP$^{Sc}$ (for Prion Protein 
Scrapie) can be
acquired either through transmission (this was the case for instance under the mad cow disease spread in the 1990s), or spontaneously, mostly above 75 years old for humans  \cite{prusiner_prions_1998}. 

Even if extensively studied in the past decades, the action of this protein on the neurons leading to a fatal issue remains unclear. However, some recent discoveries may bring possible explanations and open new therapeutic strategies. This mechanism also known as Unfolded Protein Response (or UPR) \cite{genereux_regulating_2015,hetz_disturbance_2014,hetz_er_2017,hetz_mechanisms_2020,smith_unfolded_2016} can be described as follows.

First, when produced by the cell, the PrP$^C$ proteins remain anchored to its membrane, unless misfolded PrP$^{Sc}$ in the extracellular matrix forces it to set it free and to join the pathological cohort. It is important to remind here that, by contact, a PrP$^{Sc}$
protein allows the normal form PrP$^C$ to change its conformation and to become misconformed.
Once in this state, the proteins have the ability to polymerize, that is to tie together. 
They can easily reach very large sizes, stay in the neighbourhood of the cell or diffuse in the extracellular matrix to seed other neurons (see Fig. \ref{neuron1}). 
\begin{figure}[!h]
	\begin{center}
\includegraphics[width=13cm]{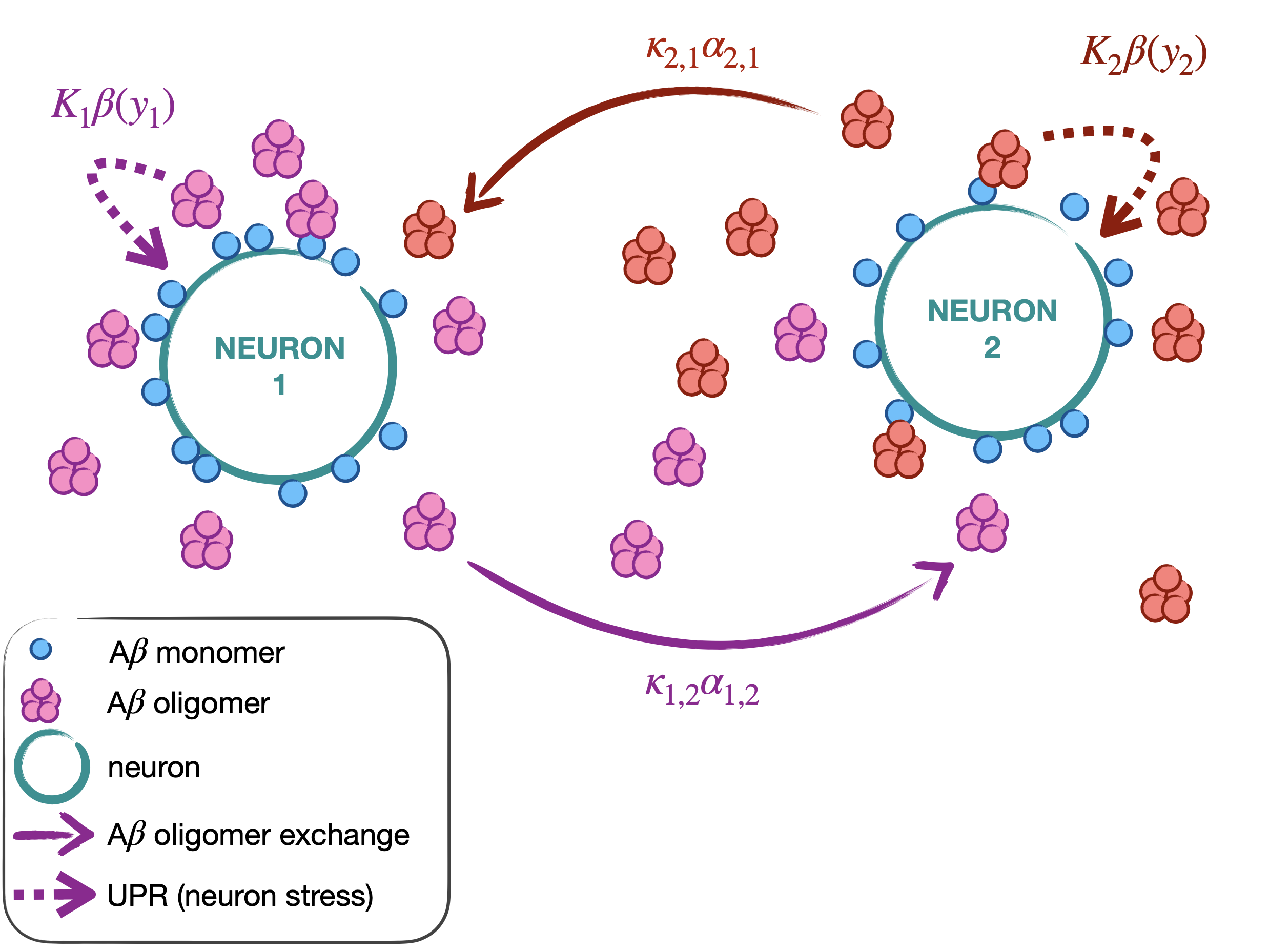}
		\caption{{\small schematic view of the PrP$^C$ protein production (in blue) by two neurons (green). The PrP$^C$ protein
  can aggregate and form pathological PrP$^{Sc}$ (pink and orange). The PrP$^{Sc}$ proteins diffuse and a certain amount can reach 
  the neighbourhood of another neuron (the orange ones can reach the neighbourhood of neuron 1, while the pink ones can reach neuron 2. We refer to Section \ref{sec:2neurons} for a complete description of this case and of the parameters and variables involved.}}
	\label{neuron1}
	\end{center}
\end{figure}

If for some reason, such as an over-expression of PrP$^C$ or a slow diffusion,  they accumulate in the neuron proximity, this latter feels it and under this induced stress shuts down almost all its activities except the vital ones. 

This global shutdown, created by a high concentration of PrP$^{Sc}$
in the neuron surrounding, causes the neuron to stop producing 
PrP$^C$, not vital for the cell (see Fig. \ref{neuron2}). 
\begin{figure}[!h]
	\begin{center}
\includegraphics[width=13cm]{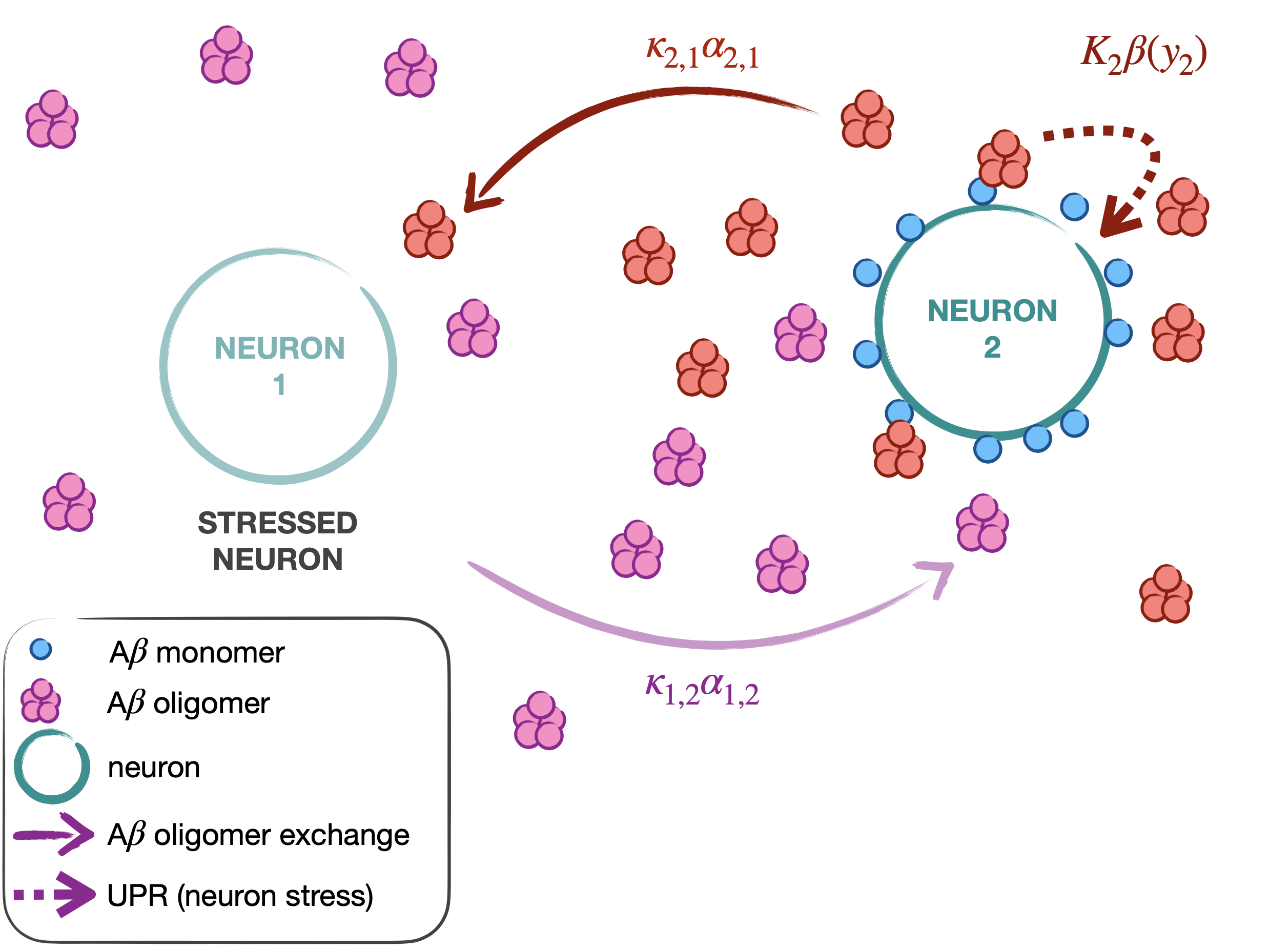}
		\caption{{\small schematic representation of a neuron (neuron 1) under Unfolded Protein Response (UPR). Stressed
  by the overcrowded amount of PrP$^{Sc}$ in its neighbourhood, neuron 1 shuts down its activities (except the vital ones). No  
  PrP$^C$ protein is then produced, and the population of PrP$^{Sc}$
  pathological proteins diffuse out the neuron surroundings.}}
	\label{neuron2}
	\end{center}
\end{figure} 

This break ends only if these 
proteins move away by diffusion or degradation. When the 
zone is clear, the cell starts again its protein production and the process continues until the next stress period.

Still under investigation, the detailed UPR mechanism remains to be
fully understood, even if several papers may be referred to the 
reader \cite{hetz_disturbance_2014,hetz_er_2017,hetz_mechanisms_2020, 
smith_unfolded_2016}. Besides, the link between UPR, PrP$^{Sc}$ has 
been put in evidence \cite{hetz_caspase-12_2003, 
tanaka_enhanced_2020,schneider_cellular_2021, 
smith_astrocyte_2020,moreno_sustained_2012, torres_prion_2011}.

Because of its complexity, the UPR \textit{modus operandi} has already been the object of mathematical models, from a gene regulatory point of view \cite{cook_knockdown_2014, schnell_model_2009,trusina_rationalizing_2008,trusina_unfolded_2010, wiseman_adaptable_2007} or through regulation of UPR intra- or extra-cellular pathways \cite{trusina_rationalizing_2008,trusina_unfolded_2010}. Our goal here is to generalise the pioneering mathematical model \cite{adimy2022neuron} dealing explicitly with prion, and investigating the parameters causing the oscillating neural activity. In \cite{adimy2022neuron}, the authors investigated the case of one neuron only, and for two neurons they gave analytical results specifically when both cells would exhibit the exact
same behaviour. In this paper, we briefly remind the model with two neurons and give new theoretical results to complete the ones of \cite{adimy2022neuron}, then we extend the construction to any neuron number $n \in \mathbb{N}$, $n\geq 2$.

%In this paper, we present a generalization of the prion model proposed in \cite{adimy2022neuron} to a network of $n$ interconnected neurons. We do so by revisiting the $2$ neurons system briefly studied in \cite{adimy2022neuron}, before generalizing the construction to any $n \in \mathbb{N}_{\geq 2}$.

We exploit the formulation of the Delay Differential Equation (DDE) system set up in detail in \cite{adimy2022neuron}, where the delay is only present in the infectious/infected variables, in order to apply a classical tool of Ordinary Differential Equations (ODE) epidemic models, namely the Next Generation Matrix. This technique was first introduced in \cite{diekmann1990definition}, then generalized in \cite{VandenDriesscheWatmough} (see also \cite{diekmann2010construction}). Through an appropriate decomposition of the Jacobian matrix evaluated at the Disease Free Equilibrium, we are able to provide a formulation for the Basic Reproduction Number $\mathcal{R}_0$ of the $n$-dimensional system, under biologically acceptable conditions.

Then, we apply the definition of the threshold quantity $\mathcal{R}_0$ to prove either global stability of the Disease Free Equilibrium (when $\mathcal{R}_0<1$) or permanence of the system (when $\mathcal{R}_0>1$). Moreover, under slightly stricter conditions, we are able to prove the existence of at least one Endemic Equilibrium.

The paper is structured as follows. In Section \ref{sec:2neurons}, we recall the $2$ neurons model introduced in \cite{adimy2022neuron}. In Section \ref{sec:n-neur}, we generalize this construction to a network of $n$ neurons, with $n \geq 2$; moreover, we show two other results: first, how the fully connected and fully homogeneous case can be qualitatively reduced to a single neuron model (but with a different $\mathcal{R}_0$) proposed in \cite{adimy2022neuron}, and second how the case of one-way direction connection of several neurons behaves like a single one.  In Section \ref{sec:DFE}, we prove global stability of the Disease Free Equilibrium when $\mathcal{R}_0<1$. In Section \ref{sec:endemeq}, we show a condition for the existence (but not uniqueness) of the Endemic Equilibrium. In Section \ref{Section 5}, we show the persistence of the system when $\mathcal{R}_0>1$. In Section \ref{sec:numer}, we provide extensive numerical simulations of the model proposed in Section \ref{sec:n-neur}. Lastly, in Section \ref{sec:concl} we conclude.

\section{System with 2 neurons}\label{sec:2neurons}

We begin by recalling the system of $2$ neurons from \cite{adimy2022neuron}. It describes the dynamics of the PrP$^C$ protein associated with neuron 1 and neuron 2, respectively $x_1$ and $x_2$, as well as the PrP$^{Sc}$ concentrations in the environment of neuron 1 and neuron 2, $y_1$ and $y_2$. Due to their biological interpretation, we only consider $x_i,y_i\geq0$. This model is represented, for $t>0$, by the following system 
\begin{equation}
\begin{split}
\frac{\mathrm{d}x_1}{\mathrm{d}t} {}={}& K_1\beta(y_1(t-T_1)) - \mu_1 x_1(t) - d x_1(t)\left(y_1(t)+ \kappa \alpha_2 y_2(t)\right),\\
\frac{\mathrm{d}x_2}{\mathrm{d}t} {}={}& K_2\beta(y_2(t-T_2)) - \mu_2 x_2(t) - d x_2(t)\left(y_2(t)+ \kappa \alpha_1 y_1(t)\right),\\
\frac{\mathrm{d}y_1}{\mathrm{d}t} {}={}&  d x_1(t)\left(y_1(t)+ \kappa \alpha_2 y_2(t)\right) -\alpha_1 y_1(t), \\
\frac{\mathrm{d}y_2}{\mathrm{d}t} {}={}&  d x_2(t)\left(y_2(t)+ \kappa \alpha_1 y_1(t)\right) -\alpha_2 y_2(t),
\end{split}
\label{eq:two_neuron_system}
\end{equation}
where $K_i > 0$ ($i=1$ or $2)$ represents the PrP$^{C}$ production rate of the neuron $i$ and $d > 0$ characterizes the force of the interaction between PrP$^{C}$ and PrP$^{Sc}$. The terms $d x_1(t)\left(y_1(t)+ \kappa \alpha_2 y_2(t)\right)$ and $d x_2(t)\left(y_2(t)+ \kappa \alpha_1 y_1(t)\right)$ stand for the new PrP$^{Sc}$ produced. The parameter $\mu_i$ represents the degradation rate of PrP$^{C}$ produced by the neuron $i$ and $\alpha_i$ is the rate at which PrP$^{Sc}$ proteins are lost through degradation or diffusion. The factor $\kappa$ indicates the interaction between proteins from different neurons. The parameter $T_i$ is the time required for a neuron $i$ to process the  PrP$^{C}$ protein synthesis.
Due to the UPR effect, increasing the amount of PrP$^{Sc}$ around a neuron decreases its activity and consequently the PrP$^{C}$ production. The contribution of PrP$^{Sc}$ concentration to PrP$^{C}$ production is therefore given through the decreasing Hill function (negative feedback \cite{adimy2022neuron})
\begin{equation}\label{eqn:bbeta}
    \beta(y)=\dfrac{1}{1+(y/y_c)^p},
\end{equation}
where $p>0$ is the sensitivity of PrP$^{Sc}$ production to PrP$^{Sc}$ overload. The parameter $y_c>0$ is the PrP$^{Sc}$ threshold beyond which the neuron stops PrP$^{C}$ production.

Compared to the notation in \cite{adimy2022neuron}, to avoid confusion we write $\beta$ instead of $\beta_n$ since $n$ will represent the number of neurons in the system from Section \ref{sec:n-neur} onward.

The Disease Free Equilibrium corresponding to the system \eqref{eq:two_neuron_system}
is
\begin{equation}\label{dfe_2}
(x_1,x_2,y_1,y_2)=\bigg( \dfrac{K_1}{\mu_1}, \dfrac{K_2}{\mu_2},0,0   \bigg).
\end{equation}
The Jacobian matrix of the system \eqref{eq:two_neuron_system} is
$$
J=\begin{pmatrix}
-\mu_1-d(y_1+\kappa\alpha_2 y_2) & 0 & K_1 \beta'(y_1)-d x_1  & -d x_1 \kappa \alpha_2 \\
0 & -\mu_2-d(y_2+\kappa\alpha_1 y_1) & -d x_2 \kappa \alpha_1  & K_2 \beta'(y_2)-d x_2 \\
d(y_1+\kappa\alpha_2 y_2) & 0 & d x_1 - \alpha_1 & d x_1 \kappa \alpha_2 \\
0 & d(y_2+\kappa\alpha_1 y_1) & d x_2 \kappa \alpha_1 &  d x_2 -\alpha_2 
\end{pmatrix}.
$$
Note that $\beta'(0)=0$. If we evaluate $J$ at the Disease Free Equilibrium \eqref{dfe_2}, we obtain
$$
J_{\text{DFE}}=\begin{pmatrix}
-\mu_1 & 0 & -d \frac{K_1}{\mu_1} & -d \kappa \alpha_2 \frac{K_1}{\mu_1} &  \\
0 & -\mu_2 & -d \kappa \alpha_1\frac{K_2}{\mu_2} &  -d\frac{K_2}{\mu_2} &  \\
0 & 0 & d \frac{K_1}{\mu_1} - \alpha_1 & d \kappa \alpha_2 \frac{K_1}{\mu_1} &  \\
0 & 0 & d \kappa \alpha_1 \frac{K_2}{\mu_2}  & d \frac{K_2}{\mu_2}-\alpha_2 
\end{pmatrix}.
$$
Now, we use the Next Generation Matrix method, firstly introduced in \cite{diekmann1990definition}, then generalized in \cite{VandenDriesscheWatmough} (see also \cite{diekmann2010construction}) to obtain the basic reproduction number $\mathcal{R}_0$ of the system \eqref{eq:two_neuron_system}. In order to do so, we need to write $J_{\text{DFE}}$ as $J_{\text{DFE}}=M-V$, with $M$ having non-negative entries and $V$ invertible. One possible choice is the following
$$
M= \begin{pmatrix}
0 & 0 & 0 & 0  \\
0 & 0 & 0 &  0   \\
0 & 0 & d \frac{K_1}{\mu_1}  & d \kappa \alpha_2 \frac{K_1}{\mu_1}   \\
0 & 0 & d \kappa \alpha_1 \frac{K_2}{\mu_2} &  d \frac{K_2}{\mu_2} 
\end{pmatrix} \quad \text{and} \quad
V=\begin{pmatrix}
\mu_1  &  0  &  d \frac{K_1}{\mu_1} &   d \kappa \alpha_2 \frac{K_1}{\mu_1}   \\
0 & \mu_2   &       d \kappa \alpha_1\frac{K_2}{\mu_2}  & d\frac{K_2}{\mu_2} &  \\
0 & 0 &  \alpha_1 & 0   \\
0 & 0 & 0 &  \alpha_2 
\end{pmatrix}.
$$
The basic reproduction number $\mathcal{R}_0$ is exactly $\rho(MV^{-1})$. Remark that, in order to compute this spectral radius, we implicitly assumed that $\alpha_1,\alpha_2 \neq 0$. This means that each neuron receives a strictly positive amount of infection from the other.
Recall that in \cite{adimy2022neuron} the basic reproduction number of neuron $i$ was computed as 
\[R_{0i}=d \frac{K_i}{\mu_i \alpha_i}.\]
Since the first two rows of $M$ are $0$, it suffices to observe the matrix
$$
F=\begin{pmatrix}
d \frac{K_1}{\mu_1 \alpha_1} &  d \kappa \frac{K_1}{\mu_1} \\
d \kappa \frac{K_2}{\mu_2} & d \frac{K_2}{\mu_2 \alpha_2}  
\end{pmatrix}=\begin{pmatrix}
 R_{01} &   \kappa \alpha_1  R_{01} \\
 \kappa \alpha_2  R_{02} &  R_{02}
\end{pmatrix},
$$
which has eigenvalues
\begin{equation}\label{eq:lam_pm}
\lambda_\pm = \dfrac{R_{01}+R_{02}\pm \sqrt{ (R_{01}-R_{02})^2+4\kappa^2 \alpha_1\alpha_2 R_{01}R_{02}}}{2},
\end{equation}
with $\lambda_+=\rho(F)$ being the new $\mathcal{R}_0$ of the $2$-neurons system. We remark that the connectivity between the two neurons $\kappa$ plays a fundamental role in the dynamics: even if both $R_{0i}<1$, with $\kappa$ large enough the disease could remain endemic. However, we recall that due to its biological interpretation, the relevant region we should consider is $\kappa \in [0,1]$.

In the next section, we generalize this construction to a network of $n\in \mathbb{N}_{\geq 2}$ neurons.

\section{System with $n$ neurons}\label{sec:n-neur}

The construction from the previous section can be generalized to a $n$ neurons case by similarly constructing the matrices $M$ and $V$. In this case, we would generally obtain $\mathcal{R}_0$ implicitly, as the spectral radius of a $2n \times 2n$ matrix. However, such a matrix can be reduced to $n \times n$ as in the previous section since $M$ will only have the lower-right quarter of non-zero entries.

We consider the following system of Delay Differential Equations (DDEs)
\begin{equation}
\begin{split}
\frac{\mathrm{d}x_i}{\mathrm{d}t} {}={}& K_i\beta(y_i(t-T_i)) - \mu_i x_i(t) - d x_i(t)\left(y_i(t)+ \sum_{j\neq i} \kappa_{ji} \alpha_{j\rightarrow i} y_j(t)\right),\\
\frac{\mathrm{d}y_i}{\mathrm{d}t} {}={}&  d x_i(t) \left(y_i(t)+ \sum_{j\neq i} \kappa_{ji} \alpha_{j\rightarrow i} y_j(t)\right) - \left(\sum_{j\neq i} \alpha_{i\rightarrow j } \right)  y_i(t). \label{eqn:nneurons}
\end{split}
\end{equation}

Due to their biological interpretation, we only consider $x_i,y_i\geq0$. The parameter $\alpha_{i\rightarrow j}$ represents the fraction of prions produced by neuron $i$ and moving towards neuron $j$. It also includes prion degradation. In other words, $\alpha_{i\rightarrow j}$ describes the diffusive property (including degradation) of PrP$^{Sc}$ to the neuron $j\neq i$. This includes both prions which die while moving away and prions which actually reach neuron $j$. The interactions between PrP$^{C}$ from neuron $i$ with PrP$^{Sc}$ of another neuron $j\neq i$ is given by the factor $\kappa_{ji}$ (it characterizes the difference between prion species); hence, $\sum_{j \neq i} \kappa_{ji}\leq 1$ for all $j$, since this sum represents the fraction of prions ``orbiting'' neuron $i$ (neuron has a number of prions it can spread to others) which does not die and manages to spread to other neurons.

For ease of notation, let 
$$\alpha_i :=\sum_{j\neq i} \alpha_{i\rightarrow j }$$
denote the total rate of migration of prions from neuron $i$, which can result in either the death of the prion or contact with any other neuron $j\neq i$.

Let  $C:=C([-T,0],\mathbb{R})$, $T:=\max_{i=1, \dots, n} T_i$, be the space of continuous functions on $[-T,0]$ and $C^+:=C([-T,0],\mathbb{R}^+)$ be the space of nonnegative continuous functions on $[-T,0]$. We assume throughout this paper that the initial conditions for the system \eqref{eqn:nneurons}, i.e. $(x_{i0},\varphi_i) \in \mathbb{R}^+\times C^+$, for $i=1,...,n$. The existence and uniqueness of nonnegative solutions of \eqref{eqn:nneurons} can be obtained by using the theory of functional differential equations.

Since the delay is discrete, the continuity of $\beta$ is sufficient to ensure the existence and uniqueness of the solution (see, \cite{Hale1993S,Kuang1993AP}). We call the \emph{history function} each function $u_t \in C$, for $t\geq 0$ and $u \in C([-T,+\infty),\mathbb{R})$ satisfying $u_t(\theta)=u(t+\theta)$ for $\theta \in [-T,0]$. Now, we show the nonnegativity and boundedness of solutions of the system \eqref{eqn:nneurons}.
\begin{proposition}\label{probound}
All solutions of the system \eqref{eqn:nneurons} with nonnegative initial conditions remain nonnegative and bounded.
\end{proposition}
\begin{proof}
We prove nonnegativity by applying the Theorem 3.4 of \cite{Smith2011SNY}. In fact, for $i=1,...,n$, if $x_i(t)=0$ then 
\[
\dfrac{\mathrm{d}x_i}{\mathrm{d}t}=K_i\beta(y_i(t-T_i)) \geq0, \quad \text{for} \ \ y_{i}\in C^+,
\]
and if $y_i(t)=0$ then
\[
\dfrac{\mathrm{d} y_i}{\mathrm{d}t}=d x_i(t)  \sum_{j\neq i} \kappa_{ji} \alpha_{j\rightarrow i} y_j(t)\geq 0, \quad \text{for} \ \  x_i\in \mathbb{R}^+, \  y_{j}\in C^+. 
\]
Then, by Theorem 3.4 of \cite{Smith2011SNY}, we get $x_i(t)\geq0$ and $y_i(t)\geq0$ for $t\geq0$.

Now, by adding both equation of $x_i$ and $y_i$, we get, for $t\geq0$,
\[
\dfrac{\mathrm{d}x_i}{\mathrm{d}t}+\dfrac{\mathrm{d}y_i}{\mathrm{d}t}=K_i\beta(y_i(t-T_i)) - \mu_i x_i(t) - \left(\sum_{j\neq i} \alpha_{i\rightarrow j } \right)  y_i(t).
\]
This implies that, for $t\geq0$,
\[
\dfrac{\mathrm{d}(x_i+y_i)}{\mathrm{d}t}\leq K_i\beta(0)-\min\{ \mu_i,\alpha_i\}(x_i+y_i).
\]
This means that
\[
\limsup_{t\rightarrow+\infty}x_i(t)+y_i(t)\leq \dfrac{K_i\beta(0)}{\min\{ \mu_i,\alpha_i\}}.
\]
Therefore, the solution should be necessarily bounded.
\end{proof}
We now introduce a formula for the Basic Reproduction Number (BRN) $\mathcal{R}_0$ of the system \eqref{eqn:nneurons}, given as the spectral radius of an $n\times n$ matrix. We do so by applying the Next Generation Matrix method \cite{diekmann1990definition,VandenDriesscheWatmough,diekmann2010construction}. We remark that this method was developed specifically for systems of ODEs. However, when evaluated in its Disease Free Equilibrium, namely
\begin{equation}\label{DFE}
   (x_1,\dots,x_n,y_1,\dots,y_n)=\bigg( \dfrac{K_1}{\mu_1},\dots, \dfrac{K_n}{\mu_n},0,\dots,0   \bigg),
\end{equation} 
the system \eqref{eqn:nneurons} does not exhibit any form of delay, and qualitatively reduces to a system of ODEs. We focus on the Jacobian on the system evaluated in this equilibrium, obtaining a reliable threshold quantity.

\begin{proposition}\label{prop:BRN}
Recall from  \cite{adimy2022neuron} that the ``Basic Reproduction Number of neuron $i$'' is $R_{0i}=d \frac{K_i}{\mu_i \alpha_i}$.

    The Basic Reproduction Number $\mathcal{R}_0$ of the system \eqref{eqn:nneurons} is given by the spectral radius of the matrix $F\in \mathbb{R}^{n \times n}$ defined as
\begin{equation}\label{eq:matrixF}
(F)_{ij}=
\begin{cases}
R_{0i} & \text{ if } j=i,\\
\kappa_{ji} \alpha_{j\rightarrow i} R_{0i} & \text{ if } j\neq i.
\end{cases}    
\end{equation}
\end{proposition}
\begin{proof}
For ease of notation, we use $\text{diag}(\cdot)$ to indicate $\text{diag}(\cdot)_{1\leq i \leq n}$, since all the diagonal matrices we consider are of dimension $n \times n$.

We compute the Jacobian $J$ of \eqref{eqn:nneurons}, dropping the explicit dependence on $(t)$ everywhere for ease of notation. By ``splitting'' the system into $x$ and $y$, we can write
$$
J=\begin{pmatrix}
 J_{11} &   J_{12} \\
 J_{21} &   J_{22}
\end{pmatrix},
$$
where %\abdennasser{(shall we use this notation \text{diag}\left( (a_i)_{1\leq i \leq n}\right))$ ?}
$$
J_{11}=\text{diag}\left(-\mu_i-d\left(y_i+ \sum_{j\neq i} \kappa_{ji} \alpha_{j\rightarrow i} y_j\right)\right), \quad (J_{12})_{ij}=
\begin{cases}
K_i \beta'(y_i)-dx_i & \text{ if } j=i,\\
-dx_i \kappa_{ji}\alpha_{j\rightarrow i} & \text{ if } j\neq i,
\end{cases}
$$
$$J_{21}=\text{diag}\left(d \left(y_i+\sum_{j\neq i} \kappa_{ji} \alpha_{j\rightarrow i} y_j\right)\right) \quad \text{and} \quad (J_{22})_{ij}=
\begin{cases}
dx_i -\alpha_i& \text{ if } j=i,\\
dx_i \kappa_{ji}\alpha_{j\rightarrow i} & \text{ if } j\neq i.
\end{cases}
$$
%$$J_{21}=\text{diag}\left( d\left( \right. y_i+ \sum_{j\neq i} \kappa_{ji} \alpha_{j\rightarrow i} y_j\abdennasser{\left.\right)}\right),$$
We now evaluate the Jacobian in the Disease Free Equilibrium of System \eqref{eqn:nneurons}, given in \eqref{DFE}.

Recall that $\beta'(0)=0$. We obtain
$$
J_{11,\text{DFE}}=\text{diag}\left(-\mu_i\right), \qquad J_{21,\text{DFE}}=0,
$$
$$
(J_{12,\text{DFE}})_{ij}=
\begin{cases}
-d\dfrac{K_i}{\mu_i} & \text{ if } j=i,\vspace{0.1cm}\\
-d\dfrac{K_i}{\mu_i} \kappa_{ji}\alpha_{j\rightarrow i} & \text{ if } j\neq i,
\end{cases} \quad \text{and} \quad (J_{22,\text{DFE}})_{ij}=
\begin{cases}
d\dfrac{K_i}{\mu_i} -\alpha_i& \text{ if } j=i,\vspace{0.1cm}\\
d\dfrac{K_i}{\mu_i} \kappa_{ji}\alpha_{j\rightarrow i} & \text{ if } j \neq i.
\end{cases}
$$
%$$(J_{22,\text{DFE}})_{ij}= \begin{cases} d\dfrac{K_i}{\mu_i} -\alpha_i& \text{ if } i=j,\vspace{0.1cm}\\ d\dfrac{K_i}{\mu_i} \kappa_{ji}\alpha_{j\rightarrow i} & \text{ if } i\neq j. \end{cases}$$
Finally, we decompose $J_{\text{DFE}}=M-V$, with
$$
M=\begin{pmatrix}
 0 &   0 \\
 0 &   M_{22}
\end{pmatrix} \quad \text{and} \quad 
V=\begin{pmatrix}
 V_{11} &   V_{12} \\
 0 &   V_{22}
\end{pmatrix},
$$
where
$$
(M_{22})_{ij}=
\begin{cases}
d\dfrac{K_i}{\mu_i} & \text{ if } j=i,\vspace{0.1cm}\\
d\dfrac{K_i}{\mu_i} \kappa_{ji}\alpha_{j\rightarrow i} & \text{ if } j\neq i,
\end{cases}
$$
and
$$
V_{22}=\text{diag}\left(\alpha_i\right).
$$
We do not write $V_{11}$ and $V_{12}$ explicitly, since they are not needed for our computations. Then, the basic reproduction number of the whole system is the spectral radius $\mathcal{R}_0=\rho(M_{22}V_{22}^{-1})=\rho(F)$, with $F\in \mathbb{R}^{n \times n}$ defined as
$$
(F)_{ij}=
\begin{cases}
R_{0i} & \text{ if } j=i,\\
\kappa_{ji} \alpha_{j\rightarrow i} R_{0i} & \text{ if } j\neq i.
\end{cases}
$$
\end{proof}
Notice that, in order to compute this spectral radius, we implicitly assumed that $\alpha_i \neq 0$ for all $i=1,2,\dots,n$. This means that we assume that each neuron receives some infection from \emph{at least} one of its neighbours. We comment more on this in Section \ref{sec:numer}. We derived the Basic Reproduction Number of System \eqref{eqn:nneurons} similarly to how we proceeded on page 2 for the 2 neurons case. Here, the influence of the various $\kappa_{ji}$ is less obvious, and we shall investigate it numerically, except for the case $\mathcal{R}_0<1$, for which we analytically prove global convergence towards the Disease Free Equilibrium \eqref{DFE} in Section \ref{sec:DFE}.

\subsection{Fully homogeneous case}\label{FHC}

Recall that we are interested in the number of neurons $n\geq 2$, so the divisions we make in this section by $n-1$ are not problematic. Assume now that the system is fully homogeneous, and that all the neurons are connected to each other. This is clearly an unrealistic setting, however, it is instructive to obtain an intuition of what the role of $n$, the number of neurons, is in the spread of the prion.

Full homogeneity in this setting means that in the system \eqref{eqn:nneurons} the parameters are $K_i=K$, $\mu_i=\mu$ for all $i=1,\dots,n$, and $\kappa_{ij}=\kappa$, 
%\mostafa{we have to discuss this choice of $\kappa_{ij}$. Maybe, we have to choose $\kappa_{ij}=\kappa$.} \mattia{I still think that we should considerthe quantity $\kappa$ as a fixed amount; hence, when dividing it equally among $n-1$ neighbours, the division by $n-1$. We say right after (5) that $\sum k_{ij} \leq 1$, so if we take it without the division by $n-1$, for some $n$ this condition clearly can not be satisfied}
$\alpha_{i \to j}=\alpha/(n-1)$ for all $i,j=1,\dots, n$. Then, $\alpha_i=\alpha$; moreover, for each neuron the local Basic Reproduction Number is 
$$R_{0i}=\frac{dK}{\mu \alpha}=:R_0,$$ 
and the matrix $F$ defining the global Basic Reproduction Number $\mathcal{R}_0$ is given by
\begin{equation}\label{eq:fullyhom}
(F)_{ij}=
\begin{cases}
R_0 & \text{ if } j=i\\
\dfrac{\kappa\alpha}{n-1} R_0 & \text{ if } j\neq i
\end{cases}=R_{0}
\begin{cases}
1 & \text{ if } j=i\\
\dfrac{\kappa\alpha}{n-1} & \text{ if } j\neq i
\end{cases}=
R_{0}\left( \left( 1 - \dfrac{\kappa\alpha}{n-1} \right) I_n + \dfrac{\kappa\alpha}{n-1} 1_n  \right),    
\end{equation}
where $I_n$ is the $n\times n$ identity matrix, and $1_n$ is the $n\times n$ matrix with $1$ in all its entries. 

Then, $1_n$ has one eigenvalue $n$ (its trace) and $n-1$ zero eigenvalues (since it has rank 1), whereas the matrix
$$
\left( 1 - \dfrac{\kappa\alpha}{n-1}\right) I_n,
$$
clearly has $n$ eigenvalues equal to $1 - \kappa\alpha/(n-1)$. 

Recall that, if a matrix $A$ has eigenvalues $\lambda_1, \dots,\lambda_n$, then the matrix $cI_n+bA$ has eigenvalues $c+b\lambda_1, \dots,c+b\lambda_n$, for any $b,c\in\mathbb{R}$, since any eigenvector $v$ of $A$ will also satisfy $cIv=cv$.

Hence, the sum \eqref{eq:fullyhom} (ignoring for a moment the scalar coefficient $R_{0}$ in front of the brackets) has one eigenvalue equal to $\kappa\alpha +1$ (its spectral radius) and $n-1$ eigenvalues equal to $1 - \kappa\alpha/(n-1)$. Consequently,
\begin{equation}\label{eq:rho}
\mathcal{R}_0=\rho(F)=R_{0}\left(\kappa\alpha+1\right).
\end{equation}
This value is clearly strictly greater than $R_0$, and independent on $n$. This means that, as long as the network is fully connected and fully homogeneous, the number of neurons has no direct impact on the dynamics of the system, according to our model. %is consistent with our previous result on two neurons in \cite{adimy2022neuron}. This means that even if for each neuron $R_0<1$, a network of neurons may have a negative effect on disease severity, $\mathcal{R}_0>1$. Hence, adding neurons would have a negative influence on the severity of the disease. 
%Depending on the value of \eqref{eq:rho}, this might change the stability of the Disease Free Equilibrium, which is globally asymptotically stable when $\mathcal{R}_0<1$, as we show in the following Section.

\subsection{Case of one-way direction}
Let us consider here the case where neuron $i$ is only connected to its
neighbour $i+1$ (respectively $i-1$), and only this one. This leads then to the following expressions,
\begin{equation*}
(F)_{ij}=
R_0\begin{cases}
1 & \text{ if } j=i,\\
\dfrac{\kappa\alpha}{n-1} & \text{ if } j < i, \\
0 & \text{ otherwise}.
\end{cases} \quad \text{or} \quad (F)_{ij}=
R_0\begin{cases}
1 & \text{ if } j=i,\\
\dfrac{\kappa\alpha}{n-1} & \text{ if } j > i, \\
0 & \text{ otherwise}.
\end{cases}
\end{equation*}
In this case, we obtain 
\[\mathcal{R}_0=\rho(F)= R_{0}. \]
The only parameters involved in the expression \eqref{eq:matrixF} are the local Basic Reproduction Number $R_{0i}$, the diffusion coefficient $\alpha_{j\rightarrow i}$ and the interaction parameter $\kappa_{ji}$. Neurons can be grouped into collections of neurons, for which the corresponding parameters $R_{0i}$, $\alpha_{j\rightarrow i}$ and $\kappa_{ji}$ can be determined. We then consider each collection of neurons as a single neuron and again apply the previously established approach. 
%\mattia{This would not however correspond to the line network; correct? It would look more like a pyramid}. \mostafa{Yes, it's true}.

\section{Global stability of the Disease Free Equilibrium}\label{sec:DFE}

In this section, we prove the global stability of the Disease Free Equilibrium \eqref{DFE} for System \eqref{eqn:nneurons} when $\mathcal{R}_0=\rho(F)<1$. In order to do so, we proceed similarly to \cite[Thm. 5]{ottaviano2022global}.

\begin{theorem}\label{DFEstab}
The Disease Free Equilibrium \eqref{DFE} of System \eqref{eqn:nneurons} is globally asymptotically stable when $\mathcal{R}_0=\rho(F)<1$.
\end{theorem}

\begin{proof}
Recall from \eqref{eqn:bbeta} that $\beta(x)\leq 1$ for all $x\geq 0$. Then, we can bound the first $n$ DDEs of System \eqref{eqn:nneurons} from above by
$$
\frac{\mathrm{d}x_i}{\mathrm{d}t} {}={} K_i\beta(y_i(t-T_i)) - \mu_i x_i(t) - d x_i\left(y_i(t)+ \sum_{j\neq i} \kappa_{ji} \alpha_{j\rightarrow i} y_j(t)\right)\leq K_i-\mu_i x_i(t).
$$
Consider the auxiliary system
\begin{equation}\label{aux1}
    \frac{\mathrm{d}z_i}{\mathrm{d}t} {}=K_i-\mu_i z_i(t), \quad i=1,2,\dots,n.
\end{equation}
Clearly, the first $n$ entries of Disease Free Equilibrium \eqref{DFE} form a point which is globally asymptotically stable for \eqref{aux1}. Then, for any $\varepsilon>0$, there exists a $\bar{t}_i>0$ such that, for $t\geq \bar{t}_i$,
$$
x_i(t)\leq \dfrac{K_i}{\mu_i}+\varepsilon.
$$
Take $\bar{t}=\max_i \bar{t}_i$. Then, for $t\geq \bar{t}$, the second $n$ ODEs of the system \eqref{eqn:nneurons} can be bound from above by
$$
\frac{\mathrm{d}y_i}{\mathrm{d}t} {}\leq{}  d \bigg( \dfrac{K_i}{\mu_i}+\varepsilon \bigg)\left(y_i(t)+ \sum_{j\neq i} \kappa_{ji} \alpha_{j\rightarrow i} y_j(t)\right) - \alpha_i y_i(t).
$$
Consider the second auxiliary system 
$$
\frac{\mathrm{d}w_i}{\mathrm{d}t} {}={}  d \bigg( \dfrac{K_i}{\mu_i}+\varepsilon \bigg)\left(w_i(t)+ \sum_{j\neq i} \kappa_{ji} \alpha_{j\rightarrow i} w_j(t)\right) - \alpha_i w_i(t).
$$
This system is linear in $w=(w_1,w_2,\dots,w_n)$, and can be rewritten as 
$$
\frac{\mathrm{d}w}{\mathrm{d}t} {}=(M_{22}(\varepsilon)-V_{22})w,
$$
where
$$
(M_{22}(\varepsilon))_{ij}=
\begin{cases}
d\bigg( \dfrac{K_i}{\mu_i}+\varepsilon \bigg) & \text{ if } j=i,\vspace{0.1cm}\\
d\bigg( \dfrac{K_i}{\mu_i}+\varepsilon \bigg) \kappa_{ji}\alpha_{j\rightarrow i} & \text{ if } j\neq i,
\end{cases}
$$
meaning the matrix $M_{22}$ used in the definition of $\mathcal{R}_0$ is actually $M_{22}(0)$, and
$$
V_{22}=\text{diag}\left(\alpha_i \right),
$$
as above.
For $\varepsilon>0$ small enough, as a consequence of our assumption $\mathcal{R}_0<1$, we can have $\rho(M_{22}(\varepsilon)V_{22}^{-1})<1$. We then use the following lemma:
\begin{lemma}[\cite{van2008further}, Lemma 2]\label{eigenlemma}
If $M$ is non-negative and $V$ is a non-singular M-matrix, then $\mathcal{R}_0=\rho(MV^{-1})<1$
if and only if all eigenvalues of $(M-V)$ have negative real parts.
\end{lemma}
This means that, if $\rho(M_{22}(\varepsilon)V_{22}^{-1})<1$, then 
$$\lim_{t\rightarrow +\infty}w_i(t)=0$$
for all $i=1,2,\dots,n$, which implies $$\lim_{t\rightarrow +\infty}y_i(t)=0.$$
Thus, for any $\delta>0$, there exists $t^*>0$ such that, for all $t\geq t^*$ and for all $i=1,2,\dots,n$, we have $y_i(t)\leq \delta$. Hence, introducing for ease of notation $T=\max_i T_i$, for $t\geq t^*+T$, we have
$$
\beta(y_i(t-T_i))=\dfrac{1}{1+(y_i(t-T_i)/y_c)^p}\geq \dfrac{1}{1+(\delta/y_c)^p}=\beta(\delta).
$$
Notice that $\beta(\delta) \rightarrow 1$ as $\delta \rightarrow 0$. 
We can then bound the first $n$ DDEs of System \eqref{eqn:nneurons} from below by
$$
\frac{\mathrm{d}x_i}{\mathrm{d}t} {}\geq {} K_i\beta(\delta) - \mu_i x_i(t) - d x_i(t)\left(\delta+ \sum_{j\neq i} \kappa_{ji} \alpha_{j\rightarrow i} \delta\right).
$$
Consider the final auxiliary system
$$
\frac{\mathrm{d}v_i}{\mathrm{d}t} {}= {} K_i\beta(\delta) - \mu_i v_i(t) - d v_i(t)\left(\delta+ \sum_{j\neq i} \kappa_{ji} \alpha_{j\rightarrow i} \delta\right).
$$
Clearly, each for each $i$ we have
$$
\lim_{t \rightarrow +\infty} v_i(t)=\dfrac{K_i \beta(\delta)}{\mu_i+\delta d(1+ \sum_{j\neq i} \kappa_{ji} \alpha_{j\rightarrow i} )}.
$$
Hence, for each $i=1,2,\dots,n$ and for all $\varepsilon,\delta>0$, we have the following lower and upper bounds: 
$$
\dfrac{K_i \beta(\delta)}{\mu_i+\delta d(1+ \sum_{j\neq i} \kappa_{ji} \alpha_{j\rightarrow i} )}\leq \liminf_{t\rightarrow+\infty} x_i(t)\leq \limsup_{t\rightarrow+\infty} x_i(t)\leq \dfrac{K_i}{\mu_i}+\varepsilon.
$$ 

Letting $\varepsilon,\delta\rightarrow 0$ concludes the proof.

\end{proof}

\begin{corollary}
    The Disease Free Equilibrium is locally unstable when $\mathcal{R}_0>1$.
\end{corollary}
\begin{proof}
Direct consequence of Theorem \ref{DFEstab} and \cite[Thm. 1]{van2008further}.
\end{proof}
We remark that, for all our results thus far, the only assumptions on the function $\beta(\cdot)$ are: $\beta(0)=1$, $\beta'(0)=0$ and $\beta(x)$ decreasing in $x$. Our specific choice \eqref{eqn:bbeta} was made for consistency with \cite{adimy2022neuron} and because it appears biologically relevant. However, other choices might lead to interesting results. We comment more on this in Section \ref{sec:concl}.

\section{Existence of an endemic equilibrium}\label{sec:endemeq}

We now prove, under stronger assumptions than $\mathcal{R}_0>1$ (but weaker than $R_{0i}>1$), that System \eqref{eqn:nneurons} admits at least one Endemic Equilibrium (EE), i.e. an equilibrium such that $y_i > 0$ for all $i$.

\begin{theorem}\label{thm:endeq}
Assume that the matrix $F$ \eqref{eq:matrixF} is such that the minimum row sum is strictly bigger than 1. Then, System \eqref{eqn:nneurons} admits at least one Endemic Equilibrium.
\end{theorem}
\begin{proof}
We know that
$$
\min \text{ row/column sum of }F \leq \rho(F) \leq \max \text{ row/column sum of }F,
$$
hence under our assumption, $\rho(F)=\mathcal{R}_0>1$.

We begin by noticing that an equilibrium of the system \eqref{eqn:nneurons} necessarily satisfies
\begin{equation}
    \label{eq:proof1}
x_i=\dfrac{K_i\beta(y_i)}{\mu_i + d \left(y_i+ \sum_{j\neq i} \kappa_{ji} \alpha_{j\rightarrow i} y_j\right)}.
\end{equation}
Substituting \eqref{eq:proof1} in the ODEs for $y_i$ and equating them to 0, we obtain 
\begin{equation}
    \label{eq:proof2}
0 =  \dfrac{dK_i\beta(y_i) \left(y_i+ \sum_{j\neq i} \kappa_{ji} \alpha_{j\rightarrow i} y_j\right)} {\mu_i + d \left(y_i+ \sum_{j\neq i} \kappa_{ji} \alpha_{j\rightarrow i} y_j\right)}- \alpha_i  y_i. 
\end{equation}
Notice that, for $y_i$ large enough, the right hand side (RHS) of \eqref{eq:proof2} is clearly negative. Let us denote with $M_i$ a large number such that 
$$
\left(  \dfrac{dK_i\beta(y_i) \left(y_i+ \sum_{j\neq i} \kappa_{ji} \alpha_{j\rightarrow i} y_j\right)} {\mu_i + d \left(y_i+ \sum_{j\neq i} \kappa_{ji} \alpha_{j\rightarrow i} y_j\right)}- \alpha_i  y_i \right)\bigg|_{y_i=M_i}<0,
$$
for all non-negative values of $y_j$, $j\neq i$. Moreover, let us denote with $M=\max M_i$.

If we find a value $\varepsilon>0$ such that the RHS of \eqref{eq:proof2} is positive for all $i=1,2,\dots,n$, we can apply the Poincaré-Miranda theorem \cite{kulpa1997poincare,mawhin2019simple} (qualitatively, a higher dimensional version of the intermediate value theorem) to conclude the existence of \emph{at least} one Endemic Equilibrium of System \eqref{eqn:nneurons}. 

Let us evaluate the RHS of \eqref{eq:proof2} at $y_i=\varepsilon$ for all $i=1,2,\dots,n$, and study its sign. We have
$$
\dfrac{dK_i\beta(\varepsilon) \left(\varepsilon+ \varepsilon\sum_{j\neq i} \kappa_{ji} \alpha_{j\rightarrow i} \right)} {\mu_i + d \left(\varepsilon+ \varepsilon\sum_{j\neq i} \kappa_{ji} \alpha_{j\rightarrow i}\right)}- \alpha_i \varepsilon>0. 
$$
In fact, we can divide by $\varepsilon>0$ on both sides, obtaining  
\begin{equation}
    \label{eq:proof3}
\dfrac{dK_i\beta(\varepsilon) \left(1+ \sum_{j\neq i} \kappa_{ji} \alpha_{j\rightarrow i} \right)} {\mu_i + d \left(\varepsilon+ \varepsilon\sum_{j\neq i} \kappa_{ji} \alpha_{j\rightarrow i}\right)}- \alpha_i >0. 
\end{equation}
Recall that $\beta(0)=1$. Then, for $\varepsilon=0$, \eqref{eq:proof3} coincides with the $i$-th row sum of $F$ being strictly greater than $1$. Since by assumption the minimum of the row sums (hence, all the row sums) is greater than $1$, by continuity there exists a small $\varepsilon_i>0$ such that the RHS of \eqref{eq:proof2} is strictly positive. Let us denote with $\varepsilon=\min \varepsilon_i$.

Applying the Poincaré-Miranda theorem on the set $[\varepsilon,M]^n$ allows us to conclude the existence of at least one Endemic Equilibrium, i.e. with $0<\varepsilon<y_i<M$ for $i=1,2,\dots,n$.
\end{proof}

We conjecture the following, based on our extensive numerical simulations:
\begin{conjecture}
  System \eqref{eqn:nneurons} admits at least one Endemic Equilibrium when $\mathcal{R}_0>1$. 
\end{conjecture}
Our proof of Theorem \ref{thm:endeq} relies heavily on the assumption on the minimum row sum being strictly bigger than 1, hence it fails so for a generic matrix $F$, if we only assume $\rho(F)>1$. However, the application of the Poincaré-Miranda theorem might not be necessary to prove this result.

	\section{Persistence of solutions}\label{Section 5}

In this section, we treat the long-term behavior of the system \eqref{eqn:nneurons} when $\mathcal{ R}_0 >1$. We start with the following proposition:
\begin{proposition}
	Consider a fixed $i=\tilde{i}\in \{ 1,\dots, n\}$. If $R_{0\tilde{i}}=d K_{\tilde{i}}/(\mu_{\tilde{i}} \alpha_{\tilde{i}})> 1$, then $\mathcal{ R}_0 > 1$ and there exists a constant $\varepsilon_{\tilde{i}} > 0$ such that 
	\[
	\limsup_{t \to +\infty} y_{\tilde{i}}(t) > \varepsilon_{\tilde{i}}, \qquad \text{with} \quad\varphi_{\tilde{i}} \in   C([-T,0], \mathbb{R}_+), \quad \varphi_{\tilde{i}}(0)\neq0.
	\]
\end{proposition}
\begin{proof}
	By considering the second equation of System \eqref{eqn:nneurons} for $i=\tilde{i}$, we have, for $t>0$,
		\begin{equation*}
	\begin{split}
	\frac{\mathrm{d}y_{\tilde{i}}}{\mathrm{d}t} {}\geq{}&  d x_{\tilde{i}}(t) y_{\tilde{i}}(t) -  \alpha_{i}   y_{\tilde{i}}(t).
	\end{split}
	\end{equation*}
	We have also, for $t>0$,
	\begin{equation*}
	\begin{split}
	\frac{\mathrm{d}(x_{\tilde{i}}+y_{\tilde{i}})}{\mathrm{d}t} {}={}& K_{\tilde{i}}\beta(y_{\tilde{i}}(t-T_1)) - \mu_{\tilde{i}} x_{\tilde{i}}(t)-\alpha_{\tilde{i}} y_{\tilde{i}}(t).
	\end{split}
	\end{equation*}
	We suppose by contradiction that
	$\limsup_{t \to +\infty} y_{\tilde{i}}(t)  \leq \varepsilon_{\tilde{i}}$, for any small $\varepsilon_{\tilde{i}}>0$. By the boundedness of solutions, we consider $\liminf_{t\rightarrow+\infty} x_{\tilde{i}}(t)=x_{\tilde{i}\infty}$ and $\liminf_{t\rightarrow+\infty} y_{\tilde{i}}(t)=y_{\tilde{i}\infty}=0$. Then, there exists a sequence $t_k\to +\infty$ as $k\to +\infty$, such that $x_{\tilde{i}}(t_k)\rightarrow x_{\tilde{i}\infty}$, $y_{\tilde{i}}(t_k)\rightarrow 0$, $y_{\tilde{i}}'(t_k)\rightarrow 0$ and $x_{\tilde{i}}'(t_k)\rightarrow 0$. This yields to
	\[
	0\geq K_{\tilde{i}}\beta(0) -\mu_{\tilde{i}}  x_{\tilde{i}\infty}\Rightarrow x_{\tilde{i}\infty}\geq \dfrac{K_{\tilde{i}}\beta(0)}{\mu_{\tilde{i}}}=\dfrac{K_{\tilde{i}}}{\mu_{\tilde{i}}}.
	\]
	For a very large time $t$, the ODE of $y_{\tilde{i}}$ then satisfies
		\begin{equation*}
	\begin{split}
	\frac{\mathrm{d}y_{\tilde{i}}}{\mathrm{d}t} {}\geq{}&   \dfrac{d K_{\tilde{i}}}{\mu_{\tilde{i}}} y_{\tilde{i}}(t) -  \alpha_{i}   y_{\tilde{i}}(t).
	\end{split}
	\end{equation*}
	By using $R_{0\tilde{i}}=d K_{\tilde{i}}/(\mu_{\tilde{i}} \alpha_{\tilde{i}})> 1$, then $\lim_{t\rightarrow+\infty} y_{\tilde{i}}(t)=y_{\tilde{i}\infty}=+\infty$, which contradicts the hypothesis and clashes with the results derived earlier on the boundedness of solutions. 
    
    Clearly, if $R_{0\tilde{i}}> 1$, then $\mathcal{ R}_0 > 1$. This is a consequence of  Theorem \ref{DFEstab}: if $\mathcal{ R}_0 < 1$, then the solution approach zero in every $y_i$ component, which is not the case for $i=\tilde{i}$.
\end{proof}
Next, we show the weak persistence of each $y_i$, $i=1,...,n$, in the following proposition.
\begin{proposition}\label{llimsupu}
	Suppose that $\mathcal{ R}_0 > 1$. Then, there exists a constant $\varepsilon > 0$ such that, for any initial condition $(x_{i0},\varphi_i) \in   \mathbb{R}_+ \times C([-T,0], \mathbb{R}_+)  $, for $i=1,...,n$, we have 
	\[
	\limsup_{t \to +\infty} y_i(t) > \varepsilon, \quad \varphi_{i}(0)\neq0.
	\]
\end{proposition}
\begin{proof}
	We suppose by contradiction that
	$\limsup_{t \to +\infty} y_i(t)  \leq \varepsilon$, for $i=1,\dots,n$ and for any small $\varepsilon>0$. Then, there exists a sufficiently large $t_{1\varepsilon}>0$ such that $y_i(t) \leq \varepsilon$, for all $t \geq t_{1\varepsilon}$. Hence, we get for all $t \geq t_{1\varepsilon}$, 
 	\begin{equation*}
		\begin{array}{llll}
		\dfrac{\mathrm{d}x_i}{\mathrm{d}t} \geq K_i\beta(\varepsilon) - \mu_i x_i(t) - d x_i\left(\varepsilon+ \varepsilon\sum_{j\neq i} \kappa_{ji} \alpha_{j\rightarrow i} \right).
		\end{array}%
	\end{equation*}
	We denote $\liminf_{t\rightarrow+\infty} x_i(t)=x_{i\infty}$, for $i=1,\dots,n$. Then, there exists a sequence $t_m\rightarrow+\infty$ as $m\rightarrow+\infty$, such that $x_i(t_m)\rightarrow x_{i\infty}$ and $x_i'(t_m)\rightarrow 0$ (see Lemma A.14 of \cite{SmithThieme2011}). This yields
	\[
	0 \geq K_i\beta(\varepsilon) - \mu_i x_{i\infty} - d x_{i\infty}\left(\varepsilon+ \varepsilon\sum_{j\neq i} \kappa_{ji} \alpha_{j\rightarrow i} \right).
	\]
	Then, we have
	\[
	x_{i\infty}\geq \dfrac{K_i\beta(\varepsilon) }{\mu_i x_{i\infty} + d \left(\varepsilon+ \varepsilon\sum_{j\neq i} \kappa_{ji} \alpha_{j\rightarrow i} \right)}=:x_{i\varepsilon}.
	\]
Hence, for every small $\nu>0$, there exists a sufficiently large $t_{2\nu}>0$ such that, for $t\geq t_{2\nu}$,
 \[
	x_{i}(t)\geq x_{i\varepsilon}-\nu=:x_{i\varepsilon}^\nu.
	\]
 Then, for a significant large time, we get
 \[
	\frac{\mathrm{d}y_i}{\mathrm{d}t} \geq  d x_{i\varepsilon}^\nu \left(y_i(t)+ \sum_{j\neq i} \kappa_{ji} \alpha_{j\rightarrow i} y_j(t)\right) -  \alpha_{i}   y_i(t).
 \]
 As in the proof of Theorem \ref{DFEstab}, we consider the following system of ODEs:
$$
\frac{\mathrm{d}w_i}{\mathrm{d}t} {}={}  d x_{i\varepsilon}^\nu \left(w_i(t)+ \sum_{j\neq i} \kappa_{ji} \alpha_{j\rightarrow i} w_j(t)\right) - \alpha_i w_i(t).
$$
This system is linear in $w=(w_1,w_2,\dots,w_n)$, and can be rewritten as 
$$
\frac{\mathrm{d}w}{\mathrm{d}t} {}=(M_{22}(\varepsilon,\nu)-V_{22})w,
$$
where
$$
(M_{22}(\varepsilon,\nu))_{ij}=
\begin{cases}
d x_{i\varepsilon}^\nu & \text{ if } j=i,\\
d x_{i\varepsilon}^\nu\kappa_{ji}\alpha_{j\rightarrow i} & \text{ if } j\neq i,
\end{cases} \qquad V_{22}=\text{diag}\left( \alpha_i \right).
$$
Using the hypothesis that $\mathcal{ R}_0=\rho(M_{22}(0,0)V_{22}^{-1}) > 1$, we can consider $\varepsilon$ and $\nu$ sufficiently small such that
	\begin{equation*} \label{condBis}
		\mathcal{R}_0^{\varepsilon,\nu}:=\rho(M_{22}(\varepsilon,\nu)V_{22}^{-1})>1.
	\end{equation*}
We then use Lemma \ref{eigenlemma} to conclude that at least one eigenvalue of $(M_{22}(\varepsilon,\nu)-V_{22}$ has positive real part. This leads to a contradiction with the assumption $\limsup_{t \to +\infty} y_i(t)  \leq \varepsilon$, for all $i=1,...,n$. As a consequence, there exists at least one $i=\Tilde{i}$ such that
$$
\limsup_{t \to +\infty} y_{\Tilde{i}}(t)=y_{\Tilde{i}\infty}  > \varepsilon.
$$
This is sufficient to conclude the weak persistence for each $i\in\{1,\dots,n\}$. 

Suppose this is not true, and for some $i\in\{1,\dots,n\}$ and $i\neq \Tilde{i}$ we have $\limsup_{t\rightarrow+\infty} y_i(t)=0$. This means since solutions remain non-negative, that
$$\lim_{t\rightarrow+\infty} y_i(t)=0.$$ 
By the boundedness of solutions and as a consequence of Barbalat's Lemma \cite{khalil2002nonlinear,sun2023gathering}, we obtain
$$\lim_{t\rightarrow+\infty} y'_i(t)=0.$$
We can then choose a sequence $t_m\rightarrow+\infty$ as $m\rightarrow+\infty$, such that $y_{\Tilde{i}}(t_m)\rightarrow y_{\Tilde{i}\infty}$. Recall the equation of $y_i$, for $t>0$,
\[
\frac{\mathrm{d}y_i}{\mathrm{d}t} =  d x_i(t) \left(y_i(t)+ \sum_{j\neq i} \kappa_{ji} \alpha_{j\rightarrow i} y_j(t)\right) - \alpha_i y_i(t).
\]
Therefore, by letting $t_m \rightarrow +\infty$ we obtain
\[
0 \geq  d x_{i\infty}  \kappa_{i\Tilde{i}} \alpha_{\Tilde{i}\rightarrow i} y_{\Tilde{i}\infty}>0.
\]
This contradiction completes the proof. 
\end{proof}
%
%\abdennasser{I think that there is no reason to write $M_{22} $ and $V_{22}$, it is sufficient to write $M$ and $V$,is it true ?} \mattia{Indeed, we can rename them for ease of notation; but I think we can discuss notation at the very end, to be fair, so we don't risk to create confusion throughout the text}
%
%\mattia{\bf In light of these propositions, what I think would be nice to show is a situation in which all the ``internal'' $R_0$ are $<1$ \emph{but} the global $R_0$ is $>1$; this way, we would showcase how crucial the connection between the neurons is}
%
%\abdennasser{I agree with this point, it is good to explore it numerically and discuss it }
Using the boundedness of the solution (see Proposition \ref{probound}) and the fact that $\beta$ is nonincreasing, we can show easily the following result.
\begin{proposition}\label{llimsupux}
	There exists a constant $\Tilde{\varepsilon} > 0$ such that, for any initial condition $(x_{i0},\varphi_i) \in   \mathbb{R}_+ \times C([-T,0], \mathbb{R}_+)  $, for $i=1,...,n$, we have 
	\[
	\liminf_{t \to +\infty} x_i(t) > \Tilde{\varepsilon}.
	\]
\end{proposition}
Now, we can establish the following result stating the strong uniform persistence of System \eqref{eqn:nneurons} when $\mathcal{ R}_0 > 1$.
\begin{theorem}\label{}
	Suppose that $\mathcal{ R}_0 > 1$. Then, there exists a constant $\varepsilon > 0$ such that, for any initial condition $(x_{i0},\varphi_i) \in   \mathbb{R}_+ \times C([-T,0], \mathbb{R}_+)  $, for $i=1,...,n$, we have 
	\[
	\liminf_{t \to +\infty} y_i(t) > \varepsilon, \quad \varphi_{i}(0)\neq0.
	\]
\end{theorem}
The proof can be adapted from the demonstration of Theorem 1 of \cite{FreedmanMoson1990PAMS} and it follows that the uniform weak persistence implies the uniform (strong) persistence (see also Theorem 7.3 of \cite{AdimyCheClau2020MBE}).

\section{Numerical simulations}\label{sec:numer}

In this section, we provide an extensive, but not exhaustive numerical exploration of the system \eqref{eqn:nneurons}, for various values of the parameters involved. Specifically, we simulate the model \eqref{eqn:nneurons} for the four choices of networks depicted in Fig. \ref{fig:networkss}: fully connected network with $n=3$ neurons; line networks with $n=5$ and $n=9$ neurons; and ring network with $n=5$ neurons. % Our numerical simulations were initialized as follows: all neurons start from the corresponding Disease Free Equilibrium, except for neuron 1, where a small infection is present. 
Our choice of initial conditions is not exactly based on biological aspects. Other choices of initial conditions can be explored, but as the analytical study has shown, the asymptotic behavior of the solution will remain unchanged, so no further simulation is required.

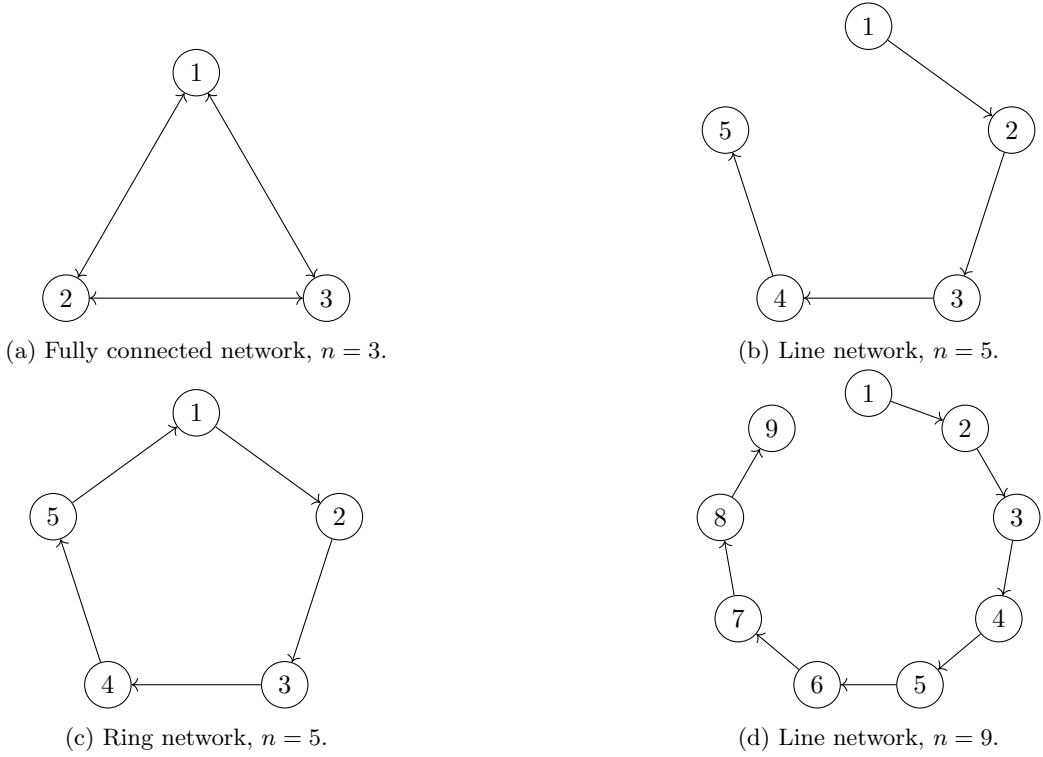
\begin{figure}[h]
    \centering
    \begin{subfigure}[t]{.49\textwidth}
      \centering

\begin{tikzpicture}
\tikzstyle{every node}=[draw,shape=circle];
\node (v1) at ( 90:2) {$1$};
\node (v2) at ( 210:2) {$2$};
\node (v3) at (330:2) {$3$};
\draw[<->] (v1) -- (v2);
\draw[<->] (v1) -- (v3);
\draw[<->] (v2) -- (v3);
\end{tikzpicture}
      \caption{Fully connected network, $n=3$.}
      \label{fig:networkssA}
\end{subfigure}
 \begin{subfigure}[t]{.49\textwidth}
      \centering

\begin{tikzpicture}
\tikzstyle{every node}=[draw,shape=circle];
\node (v1) at ( 90-72*0:2) {$1$};
\node (v2) at ( 90-72*1:2) {$2$};
\node (v3) at (90-72*2:2) {$3$};
\node (v4) at (90-72*3:2) {$4$};
\node (v5) at (90-72*4:2) {$5$};
\draw[->] (v1) -- (v2);
\draw[->] (v2) -- (v3);
\draw[->] (v3) -- (v4);
\draw[->] (v4) -- (v5);
\end{tikzpicture}
      \caption{Line network, $n=5$.}
      \label{fig:networkssB}
\end{subfigure}
 \begin{subfigure}[t]{.49\textwidth}
      \centering

\begin{tikzpicture}
\tikzstyle{every node}=[draw,shape=circle];
\node (v1) at ( 90-72*0:2) {$1$};
\node (v2) at ( 90-72*1:2) {$2$};
\node (v3) at (90-72*2:2) {$3$};
\node (v4) at (90-72*3:2) {$4$};
\node (v5) at (90-72*4:2) {$5$};
\draw[->] (v1) -- (v2);
\draw[->] (v2) -- (v3);
\draw[->] (v3) -- (v4);
\draw[->] (v4) -- (v5);
\draw[->] (v5) -- (v1);
\end{tikzpicture}
      \caption{Ring network, $n=5$.}
      \label{fig:networkssC}
\end{subfigure}
 \begin{subfigure}[t]{.49\textwidth}
      \centering

\begin{tikzpicture}
\tikzstyle{every node}=[draw,shape=circle];
\node (v1) at ( 90-40*0:2) {$1$};
\node (v2) at ( 90-40*1:2) {$2$};
\node (v3) at (90-40*2:2) {$3$};
\node (v4) at (90-40*3:2) {$4$};
\node (v5) at (90-40*4:2) {$5$};
\node (v6) at (90-40*5:2) {$6$};
\node (v7) at (90-40*6:2) {$7$};
\node (v8) at (90-40*7:2) {$8$};
\node (v9) at (90-40*8:2) {$9$};
\draw[->] (v1) -- (v2);
\draw[->] (v2) -- (v3);
\draw[->] (v3) -- (v4);
\draw[->] (v4) -- (v5);
\draw[->] (v5) -- (v6);
\draw[->] (v6) -- (v7);
\draw[->] (v7) -- (v8);
\draw[->] (v8) -- (v9);
\end{tikzpicture}
      \caption{Line network, $n=9$.}
      \label{fig:networkssD}
\end{subfigure}
 \caption{the four networks we consider in our numerical simulations. Notice that only the network with $n=3$ has double arrows on each edge, representing the fully connected network. The three remaining networks all have unidirectional edges.} \label{fig:networkss}
\end{figure}
We start by considering two fully connected, and fully homogeneous scenarios, which correspond to Fig. \ref{fig:networkssA}. 

An illustration for the case $\mathcal{R}_0<1$ with $n=3$ is given in Fig. \ref{GASR0inf1}. In this configuration, we observe an extinction of the disease, as  analytically expected and proven in Section \ref{sec:DFE}. 

In Fig. \ref{PersiR0sup1}, instead, we illustrate the case $\mathcal{R}_0>1$ with $n=3$. This figure shows that the connectivity between the neurons expressed by $\kappa$ plays a fundamental role in the dynamics in the sense that even if all $R_{0i}<1$, with $\kappa$ big enough the global system might have a Basic Reproduction Number  $\mathcal{R}_0>1$, and the disease could hence remain endemic.

\begin{figure}[!h]
	\begin{center}
\includegraphics[width=16cm]{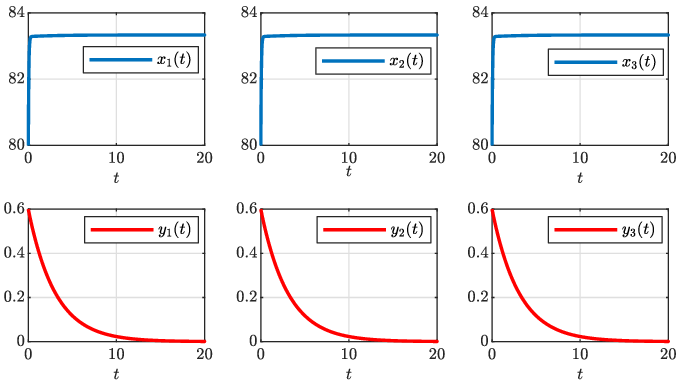}
		\caption{{\small the fully connected network of $n=3$ neurons (recall Fig. \ref{fig:networkssA}), showing the stability of disease free equilibrium when $\mathcal{R}_0=\rho(F)=0.8194<1$. In this case $R_{0i}=0.6944<1$ for $i=1,2,3$. The parameters are: $\alpha_{i\rightarrow j}=0.9$, $\kappa_{ij}=0.1$, $p=5$, $d=0.015$, $y_c=60$, $T=0.17$, $K_i=1500$ and $ \mu_i=18$.}}
	\label{GASR0inf1}
	\end{center}
\end{figure}

\begin{figure}[!h]
	\begin{center}
\includegraphics[width=16cm]{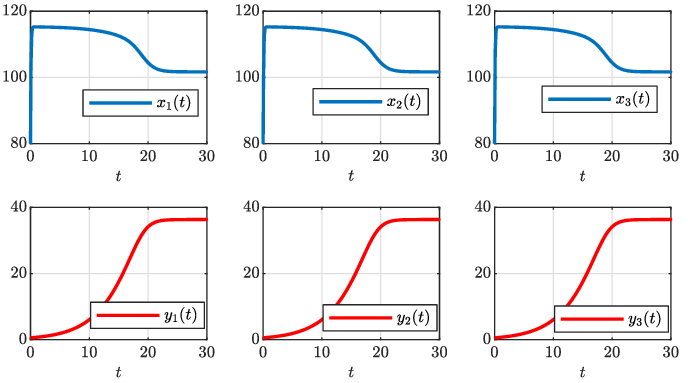}
		\caption{{\small the fully connected network of $n=3$ neurons (depicted by Fig. \ref{fig:networkssA}), showing the stability of endemic equilibrium when $\mathcal{R}_0=\rho(F)=1.1346>1$. In this case $R_{0i}=0.9615<1$ for $i=1,2,3$. The parameters are: $\alpha_{i\rightarrow j}=0.9$, $\kappa_{ij}=0.1$, $p=5$, $d=0.015$, $y_c=60$, $T=0.17$, $K_i=1500$ and $ \mu_i=13$. }}
	\label{PersiR0sup1}
	\end{center}
\end{figure}

%\begin{figure}[!h]
	%\begin{center}
%  \includegraphics[width=8cm]{twoneuronebidirection.pdf}\\
 % \includegraphics[width=8cm]{diaglinear2neurone.pdf}
	%\includegraphics[width=8cm]{linear3neurone.pdf}
	%	\caption{{\small }}
	%\label{}
	%\end{center}
%\end{figure}

\begin{figure}[!h]
	\begin{center}
\includegraphics[width=15cm]{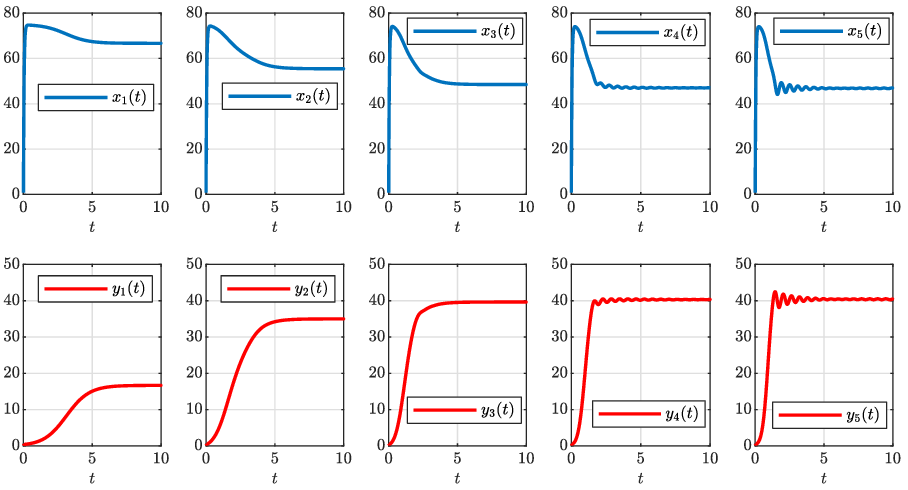}
		\caption{{\small line network of $n=5$ neurons (depicted in Fig. \ref{fig:networkssB}). This case shows that cutting the connection showed stabilization. The parameters are: $\alpha_{i\rightarrow j}=2.5$, $\kappa_{ij}=0.17$ (when considered), $p=10$, $d=0.15$, $y_c=50$, $K_i=1500$, $ \mu_i=20$ and $T=0.15$.}}
	\label{stabililine5}
	\end{center}
\end{figure}

\begin{figure}[!h]
	\begin{center}
 \includegraphics[width=15cm]{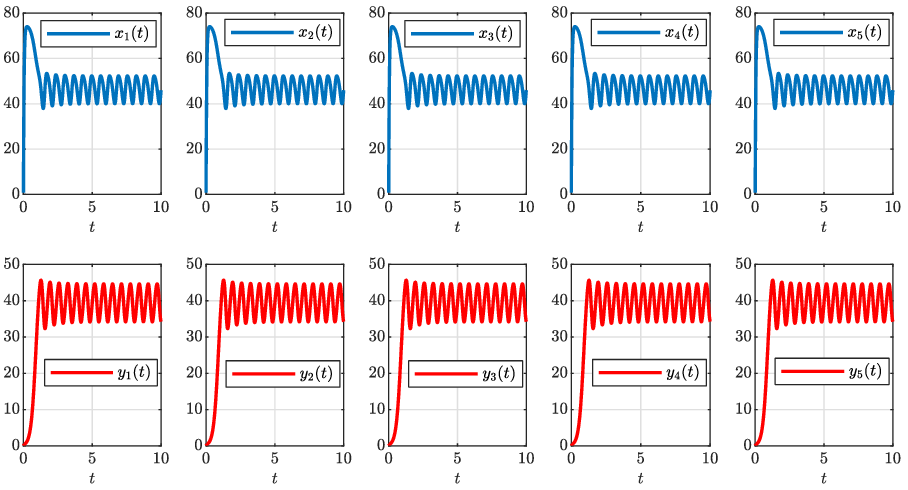}
		\caption{{\small circle (ring) network of $n=5$ neurons (depicted in Fig. \ref{fig:networkssC}). This case shows that linking the connection in Fig. \ref{stabililine5} showed destabilization of the system. The parameters are: $\alpha_{i\rightarrow j}=2.5$, $\kappa_{ij}=0.17$ (when considered), $p=10$, $d=0.15$, $y_c=50$, $K_i=1500$, $ \mu_i=20$ and $T=0.15$.}}
	\label{rinfoscilat}
	\end{center}
\end{figure}

\begin{figure}[!h]
	\begin{center}
 \includegraphics[width=16cm]{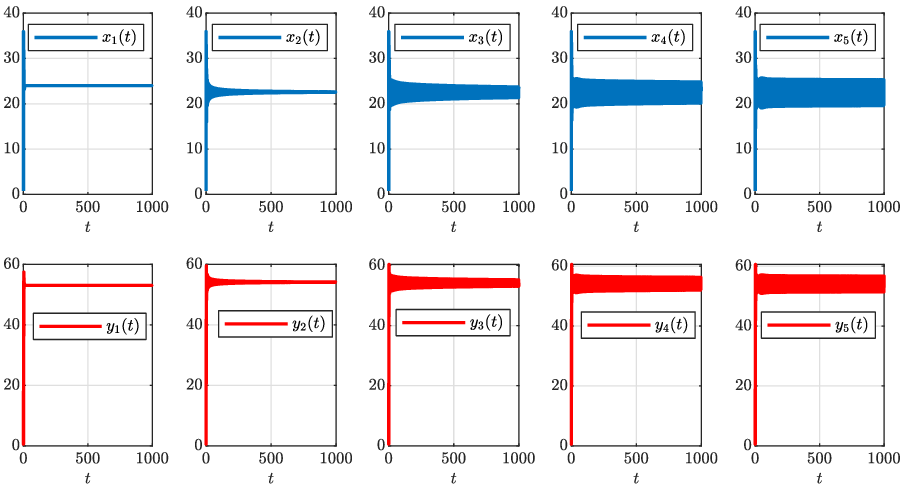}
		\caption{{\small  line network of $n=5$ neurons (depicted in Fig. \ref{fig:networkssB}). For this figure showing the oscillation of only some neurons ($n=4,5$, the last ones), we took $\kappa=0.071$. Parameters are: $p=10$, $d=0.15$, $y_c=60$, $K_i=1800$, $ \mu_i=50$, $\alpha_{i\rightarrow j}=0.9$ ($\alpha_i=3.6$) and $T=0.15$. }}
	\label{stabinstabenurone5}
	\end{center}
\end{figure}

\begin{figure}[!h]
	\begin{center}
\includegraphics[width=10cm]{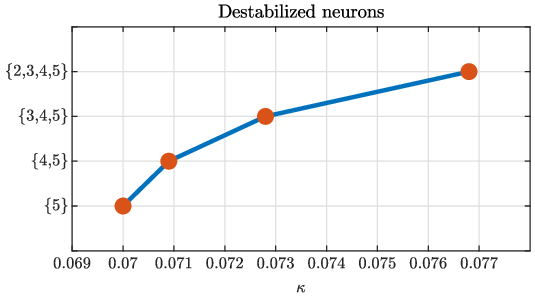}
		\caption{{\small  line network of $n=5$ neurons (depicted in Fig. \ref{fig:networkssB}). Our starting point was the case where the system is stable and we have increased the bifurcation parameter $\kappa$ and noted each value for which a neuron is destabilized. The value of bifurcations are $0.07$, $0.0709$, $0.728$ and $0.768$. The neurons are destabilized one by one from the last and going up to the second. The last neurons always look destabilized almost at the same time (approximately the same bifurcation values). Parameters are: $p=10$, $d=0.15$, $y_c=60$, $K_i=1800$, $ \mu_i=50$, $\alpha_{i\rightarrow j}=0.9$ ($\alpha_i=3.6$) and $T=0.15$.}}
	\label{diagbifura}
	\end{center}
\end{figure}

\begin{figure}[!h]
	\begin{center}
\includegraphics[width=12cm]{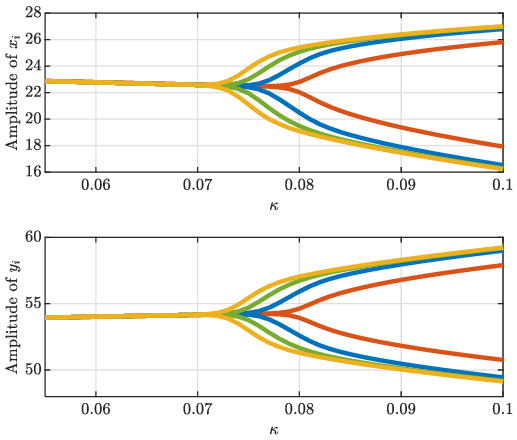}
		\caption{{\small line network of $n=5$ neurons (recall Fig. \ref{fig:networkssB}). The amplitude of the oscillations as a function of $\kappa\in[0.055,0.1]$. The amplitude of the oscillations becomes almost the same for higher values of $\kappa$. Parameters are: $p=10$, $d=0.15$, $y_c=60$, $K_i=1800$, $ \mu_i=50$, $\alpha_{i\rightarrow j}=0.9$ ($\alpha_i=3.6$) and $T=0.15$. The red, blue, green and yellow curves are associated respectively with $(x_2,y_2)$, $(x_3,y_3)$, $(x_4,y_4)$ and $(x_5,y_5)$.}}
	\label{amplitudediag}
	\end{center}
\end{figure}

\begin{figure}[!h]
	\begin{center}
 \includegraphics[width=16cm]{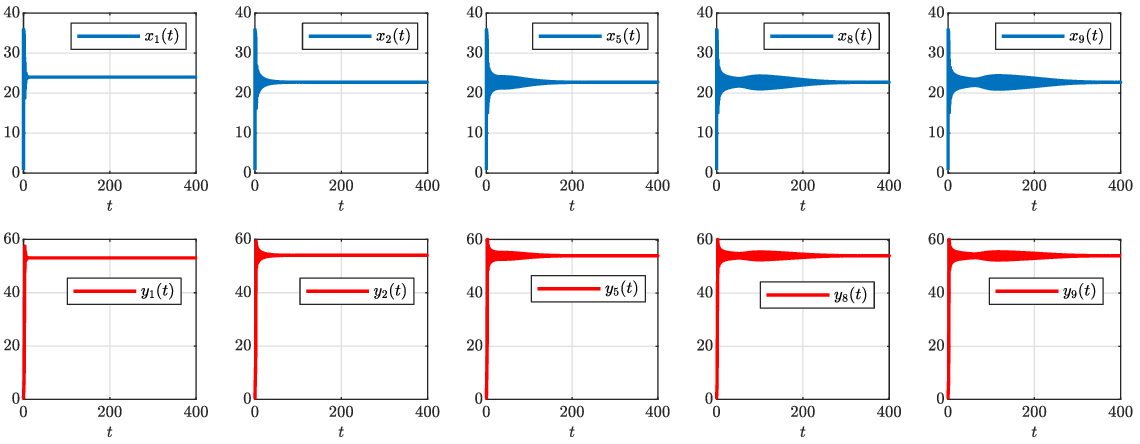}
		\caption{{\small  line network of $n=9$ neurons (recall Fig. \ref{fig:networkssD}). We only plot the time series of $(x_i,y_i)$ with $i=1,2,5,8,9$, respectively the first two, the central one and the last two neurons in the line network. By increasing the number $n$ of neurons from $5$ to $9$ with the same parameters as in Fig. \ref{diagbifura} and \ref{amplitudediag} and with $\kappa=0.125$, the system becomes stable.  Recall parameters: $p=10$, $d=0.15$, $y_c=60$, $K_i=1800$, $ \mu_i=50$, $\alpha_{i\rightarrow j}=0.45$ ($\alpha_i=3.6$) and $T=0.15$.}}
	\label{figure11}
	\end{center}
\end{figure}

We are also interested in the fully homogeneous case ($\kappa_{ij}:=\kappa$ and $\alpha_{i\rightarrow j }=\alpha/(n-1)$) with a line connection between neurons as depicted in Fig. \ref{fig:networkssB}, \ref{fig:networkssC} and \ref{fig:networkssD}. In this case, the system has the following form, for $t\geq0$,
\begin{equation*}
\begin{split}
\frac{\mathrm{d}x_1}{\mathrm{d}t} {}={}& K\beta(y_1(t-T)) - \mu x_1(t) - d x_1(t)y_1(t),\\
\frac{\mathrm{d}y_1}{\mathrm{d}t} {}={}&  d x_1(t) y_1(t) - \alpha  y_1(t),\vspace{0.1cm}\\
\frac{\mathrm{d}x_2}{\mathrm{d}t} {}={}& K\beta(y_2(t-T)) - \mu x_2(t) - d x_2(t)\left(y_2(t)+  \dfrac{\kappa \alpha}{n-1} y_1(t)\right),\\
\frac{\mathrm{d}y_2}{\mathrm{d}t} {}={}&  d x_2(t) \left(y_2(t)+\dfrac{\kappa \alpha}{n-1} y_1(t)\right)- \alpha y_2(t), \\ &\vdots \\
\frac{\mathrm{d}x_i}{\mathrm{d}t} {}={}& K\beta(y_i(t-T)) - \mu x_i(t) - d x_i(t)\left(y_i(t)+  \dfrac{\kappa \alpha}{n-1} y_{i-1}(t)\right),\\
\frac{\mathrm{d}y_i}{\mathrm{d}t} {}={}&  d x_i(t) \left(y_i(t)+\dfrac{\kappa \alpha}{n-1} y_{i-1}(t)\right)- \alpha   y_i(t),\\
&\vdots\label{syste2neuronkalphanull}
\end{split}
\end{equation*}
and the matrix \eqref{eq:matrixF} becomes
\begin{equation*}\label{eq:matrixFline}
(F)_{ij}=
R_{0}\begin{cases}
1 & \text{ if } j=i,\\
\dfrac{\kappa \alpha}{n-1} & \text{ if } j=i-1, \; i\geq 2, \\
0 & \text{ otherwise}.
\end{cases}    
\end{equation*}
Hence, we obtain
\[\mathcal{R}_0=\rho(F)=R_0=\frac{dK}{\mu \alpha}.\]
%\mattia{I don't see how you can get around the fact that $\alpha_1=0$; how is the matrix $V$ invertible, in this case?}
Note that in the non-fully homogeneous case, and non-fully connected like in Fig. \ref{fig:networkssD} $(b)$, $(c)$ or $(d)$ (where parameters can be different from one neuron to another) we cannot apply this definition of $\mathcal{R}_0$, since some of the $\alpha_i$'s may end up to be equal to $0$. We need then to go back to the more general theory. \\

In  Fig. \ref{stabililine5}, a fully homogeneous line network, after some initial ``wobbling'', the system approaches an Endemic Equilibrium. In particular, we notice that neurons further down the line (\textit{i.e.}, couples $(x_i,y_i)$ with larger $i$'s ($i=4,$ or $5$) approach equilibrium with a higher value for the infected compartment $y_i$.

In Fig. \ref{rinfoscilat} we ``close'' the line, forming a close ring with the five neurons. Without changing the other parameters, this additional link completely destabilizes the system. Each neuron approaches a limit cycle. Proving the existence of such a limit cycle analytically remains to prove and will be the object of future works.

In Fig. \ref{stabinstabenurone5}, we revisit the $n=5$ line structure, illustrating a case in which neurons 4 and 5 exhibit instability in their asymptotic behaviour, whereas neuron 1 converges to an equilibrium. This is possible because neuron 1 only spreads the infection, and does not receive feedback from the remaining neurons in the network. Moreover, we provide more detail on the parameter $\kappa$, which incorporates crucial information, namely the interconnectivity between neurons. We provide the bifurcation values of $\kappa$, assuming all the other parameters to be fixed, for which each neuron of the system destabilizes.

In Fig. \ref{diagbifura}, we explored the case of Fig. \ref{stabinstabenurone5} a little further in the following sense: beginning our simulations from the stable case, we increased $\kappa$, one of the key parameters for the Hopf bifurcations and managed to compute the exact value of $\kappa$ at which the first neuron of the line would oscillate. For instance, at $\kappa=0.07$ only the last (the fifth) one 
is destabilized, while for $\kappa=0.077$ only the first one is stable, while all the others oscillate. We observe that this process is non-linear. Predicting the number of neurons destabilized with respect to $\kappa$ analytically is also an open problem that we keep for future work.

In Fig. \ref{amplitudediag}, we showcase how a continuous variation of the parameter $\kappa$ impacts the asymptotic value of each $x_i$ and $y_i$, again for the $n=5$ line network. Increasing $\kappa$ causes more neurons to destabilize; for each of them, we plot the maximum and minimum values assumed by each variable as they asymptotically approach a limit cycle.

Finally, in Fig. \ref{figure11}, we explore the effect of increasing the number of neurons in the line network in the case of our choice of 
$\kappa_{ij}=\kappa=0.125$ set up in Section \ref{FHC}. Using similar parameter values as in Fig. \ref{stabinstabenurone5} (except for $\kappa_{ij}$ and $\alpha_{i \to j}$, to keep them biologically feasible). We manage to show that adding more neurons (going from 5 to 9 neurons here) can damp the 
oscillations, and lead the system of neurons to stabilize again. This process has to be explored biologically to be confirmed experimentally. 

We remark that our results in Section \ref{FHC} only concern $\mathcal{R}_0$, from which however we can only predict extinction or permanence of the disease, and nothing on the asymptotic stability of orbits. Hence, further analytical results in this direction are a promising research outlook.

%%%%%%%%%%%%%%%%%%%%%%%%%%%%%%%%%%%%%%%%%%%%%%%%%%%%%%%%%%%%%%%%%%%%

\section{Conclusions and outlook}\label{sec:concl}
 In this paper, we presented a model for the delayed spread of prion in a network of $n$ neurons, building on the $1$ neuron model proposed in \cite{adimy2022neuron}. We studied its analytical properties and provided extensive numerical simulations to illustrate various scenarios.

Due to the high dimension of the system, and of the analytical complexity of systems of DDEs, many questions remain unanswered. How can we overcome the requirement that $\alpha_i>0$ in the definition of $\mathcal{R}_0$? The system has a clear biological interpretation even when this condition is not satisfied; hence, we would like to find a threshold quantity in such a scenario. Moreover, does the Endemic Equilibrium exist for all systems with $\mathcal{R}_0>1$? Is it unique? If yes, when is it stable? In our numerical exploration, we found both convergences to Endemic Equilibrium and sustained oscillations, indicating a possible stable limit cycle arising in the system. It would be of interest to understand which relations between the parameters of the system lead to the former or the latter.

Finally, we should point out here also that spatial  structure has not
been taken into account. Indeed, considering diffusion may appear challenging for the following reason: in the case of Alzheimer's disease, 
oligomers diffuse randomly in the brain tissue since the A$\beta$ monomers are no longer anchored to the cell membrane. On the contrary, for the prion disease, 
 pathological PrP$^{Sc}$ proteins spread following the axon (and thus the cell membrane) where the source of non-pathological PrP$^C$ proteins 
 are attached (thanks to a GPI anchor). Thus diffusion cannot be represented in the same way depending on the neurodegenerative disease studied. Furthermore, $PrP^{Sc}$ might not be produced immediately after the contact with previous neurons. A time lag may be needed, and thus this could involve some de-synchronisation and perhaps some chaotic behaviour. 
We leave these fundamental and other interesting questions as the possible outlook for future work.\\
\\
\noindent \textbf{Acknowledgements.} Mattia Sensi was supported by the Italian Ministry for University and Research (MUR) through the PRIN 2020 project ``Integrated Mathematical Approaches to Socio-Epidemiological Dynamics'' (No. 2020JLWP23, CUP: E15F21005420006).\\ Abdennasser Chekroun was supported by the grant PRFU: C00L03UN29012022002,
from DGRSDT of Algeria.\\
Laurent Pujo-Menjouet was supported by ANR grant PrionDiff ANR-21-CE15-0011-02. 

{\footnotesize
	\bibliographystyle{unsrt}
	\bibliography{biblio}
}

\end{document}